\documentclass{amsart}

\usepackage[british]{babel}

\usepackage{amsmath}
\usepackage{amssymb}
\usepackage{amsopn}
\usepackage{amsthm}
\usepackage{tikz}
\usepackage{tikz-cd}
\usetikzlibrary{%
  matrix,%
  calc,%
  arrows%
}

\newtheorem{Thm}{Theorem}

\newtheorem{Lem}{Lemma}

\theoremstyle{definition}

\theoremstyle{remark}

\newcommand{\cpx}[1]{\ensuremath{\mathbf{#1}}}    

\newcommand{\digraph}[1]{\mathcal{G}(#1)}
\newcommand{\morsegraph}[2]{\mathcal{G}(#1)^{#2}}

\newcommand{\sigmamap}{\mathcal{S}}

\begin{document}

\author{Emil Sk\"oldberg}
\title{Algebraic Morse theory and homological perturbation theory}
\address{School of Mathematics, Statistics and Applied Mathematics\\
         National University of Ireland, Galway\\
         Ireland}

\email{emil.skoldberg@nuigalway.ie}
\urladdr{http://www.maths.nuigalway.ie/~emil/}

\date{\today}

\begin{abstract}
  We show that the main result of algebraic Morse theory can be
  obtained as a consequence of the perturbation lemma of Brown
  and Gugenheim.
\end{abstract}


\maketitle

\section{Introduction}

Robin Forman introduced discrete Morse theory in \cite{forman:morse}
as a combinatorial adaptation of the classical Morse theory suited for
studying the topology of CW-complexes.
Its fundamental idea is also applicable in purely algebraical situations
(see e.g.\ \cite{jonsson:graphs}, \cite{kozlov:morse},
\cite{joellenbeck_welker:morse}, \cite{skoldberg:morse}).

Homological perturbation theory on the other hand builds on the
perturbation lemma \cite{brown:ezthm}, \cite{gugenheim:ccpx}.
In addition to its applications in algebraic topology, it has
also found uses in e.g.\ the study of group
cohomology \cite{lambe:hh}, \cite{huebschmann:nilpotent_groups},
resolutions in commutative
algebra \cite{johansson_lambe_skoldberg:resolutions} as well
as in operadic settings, \cite{berglund:perturbation_theory_operads}.

In this note we show how to derive the main result of algebraic
Morse theory from the perturbation lemma. In related work, Berglund
\cite{berglund:hpt_morse}, has also treated connections between
algebraic Morse theory and homological perturbation theory.

\section{Definitions}

We will briefly review the definitions of the main objects of study.


A \emph{contraction} is a diagram of chain complexes of (left or right)
modules over a ring $R$
\begin{center}
  \begin{tikzpicture}
    \matrix (n) [matrix of math nodes, column sep=2em,
    nodes={asymmetrical rectangle}]
    { \cpx{D} & \cpx{C} \\ };

    \path[-stealth] ($(n-1-1.east)+(0,0.1)$) edge node [auto] {$g$}
    ($(n-1-2.west)+(0,0.1)$);

    \path[-stealth] ($(n-1-2.west)-(0,0.1)$) edge node [auto] {$f$}
    ($(n-1-1.east)-(0,0.1)$);

    \draw[-stealth] (n-1-2) to [out=30, in=330, looseness=5]
    node [auto] {$h$} (n-1-2);
  \end{tikzpicture}
\end{center}
where $f$ and $g$ are chain maps and $h$ is a degree 1 map satisfying
the identities
\[
fg = 1,\quad gf = 1 + dh + hd
\]
and
\[
fh = 0, \quad  hg = 0, \quad h^2 = 0.
\]
A contraction is \emph{filtered} if there is a bounded
below exhaustive filtration on the complexes which is preserved by
the maps $f$, $g$ and $h$.
A \emph{perturbation} of a chain complex $\cpx{C}$ is a map
$t: \cpx{C} \rightarrow \cpx{C}$ of degree $-1$ such that
$(d+t)^2 = 0$. Given a perturbation $t$ on $\cpx{C}$, we let
$\cpx{C}^t$ be the complex obtained by equipping $\cpx{C}$
with the new differential $d+t$.

We can now state the perturbation lemma.
\begin{Thm}[Brown, Gugenheim]
  Given a filtered contraction
  \begin{center}
    \begin{tikzpicture}
      \matrix (n) [matrix of math nodes, column sep=2em,
      nodes={asymmetrical rectangle}]
      { \cpx{D} & \cpx{C} \\ };

      \path[-stealth] ($(n-1-1.east)+(0,0.1)$) edge node [auto] {$g$}
        ($(n-1-2.west)+(0,0.1)$);

      \path[-stealth] ($(n-1-2.west)-(0,0.1)$) edge node [auto] {$f$}
        ($(n-1-1.east)-(0,0.1)$);

      \draw[-stealth] (n-1-2) to [out=30, in=330, looseness=5]
        node [auto] {$h$} (n-1-2);
  \end{tikzpicture}
  \end{center}
  and a filtration lowering perturbation $t$ of $\cpx{C}$, the diagram
  \begin{center}
    \begin{tikzpicture}
      \matrix (n) [matrix of math nodes, column sep=2em,
      nodes={asymmetrical rectangle}]
      { \cpx{D}^{t'} & \cpx{C}^{t} \\ };

      \path[-stealth] ($(n-1-1.east)+(0,0.1)$) edge node [auto] {$g'$}
        ($(n-1-2.west)+(0,0.1)$);

      \path[-stealth] ($(n-1-2.west)-(0,0.1)$) edge node [auto] {$f'$}
        ($(n-1-1.east)-(0,0.1)$);

      \draw[-stealth] (n-1-2) to [out=30, in=330, looseness=5]
        node [auto] {$h'$} (n-1-2);
  \end{tikzpicture}
  \end{center}
  where
  \[
  f' = f + f\sigmamap h, \quad
  g' = g + h\sigmamap g, \quad
  h' = h + h\sigmamap h, \quad
  t' = f \sigmamap g
  \]
  and
  \[
  \sigmamap = \sum_{n=0}^{\infty}t(ht)^n
  \]
  defines a contraction.
\end{Thm}



Let us next review some terminology of algebraic Morse theory.
By a \emph{based complex} of $R$-modules we mean a chain complex
$\cpx{C} $ of  $ R $-modules together with direct sum
decompositions  $ C_n = \bigoplus_{\alpha \in I_n} C_{\alpha}$ where
 $\{I_n\} $ is a family of  mutually disjoint index sets.
For $f: \bigoplus_n C_n \rightarrow \bigoplus_n C_n$ a graded map, we write
$f_{\beta,\alpha}$ for the component of $f$ going from
$C_{\alpha}$ to $C_{\beta}$, and given a based complex $\cpx{C}$ we
construct  a digraph $ \digraph{\cpx{C}} $ with vertex set
$V = \bigcup_n I_n$ and with a directed  edge
$\alpha \rightarrow \beta$ whenever the component
$d_{\beta,\alpha}$ is non-zero.

A subset $M$ of the edges of $\digraph{\cpx{C}}$ such that no
vertex is incident to more than one edge of $M$ is called
a \emph{Morse matching} if, for each edge
$\alpha \rightarrow \beta$ in $M$, the corresponding component
$d_{\beta,\alpha}$ is an isomorphism, and furthermore there is a well
founded partial order $\prec$ on each $I_n$ such that
$\gamma \prec \alpha$ whenever there is a path
$\alpha^{(n)} \rightarrow \beta \rightarrow \gamma^{(n)}$
in the graph $\morsegraph{\cpx{C}}{M}$, which is the graph obtained
from $\digraph{\cpx{C}} $ by reversing the edges from $M$.


Given the matching $M$, we define the set  $M^{0}$ to be the vertices
that are not incident to an arrow from $M$.

For $\alpha$ and $\beta$ vertices in $\morsegraph{\cpx{C}}{M}$ we can
now consider all directed paths from $\alpha$ to $\beta$. For each
such path $\gamma$, we get a map from $C_{\alpha}$ to $C_{\beta}$ by, for each
edge $\sigma \rightarrow \tau$ in $\gamma$ which is not in $M$ take the map
$d_{\tau,\sigma}$, and for each edge $\sigma \rightarrow \tau$ in
$\gamma$ which
is the reverse of an edge in $M$ take the map $-d_{\sigma,\tau}^{-1}$
and composing them. Summing these maps over all paths from $\alpha$ to
$\beta$ defines the map
$\Gamma_{\beta,\alpha}:C_{\alpha} \rightarrow C_{\beta}$.

\section{The main result}

From the based complex $\cpx{C}$ with $C_n = \bigoplus_{\alpha \in I_n} C_{\alpha}$
furnished with a Morse matching $M$, we define another based complex
$\cpx{\tilde{C}}$ by letting it be isomorphic to $\cpx{C}$ as
a graded module, and defining the differential
$\tilde{d}$ in $\cpx{\tilde{C}}$ as
\[
\tilde{d} (x) =
\begin{cases}
  d_{\beta,\alpha}(x), & \quad \text{if } \alpha \rightarrow \beta \in M, \\
  0, & \quad \text{otherwise};
\end{cases}
\qquad \text{for } x \in C_{\alpha}.
\]

We also need a based complex complex coming from the vertices in $M^0$,
so we define $\cpx{\tilde{C}}^{M}$ by
\[
\tilde{C}^{M}_n = \bigoplus_{\alpha \in I_n \cap M^{0}} C_{\alpha},
\quad d_{\cpx{\tilde{C}}^{M}} = 0,
\]
and maps
$\tilde{f}: \cpx{\tilde{C}} \rightarrow \cpx{\tilde{C}}^{M}$,
$\tilde{g}: \cpx{\tilde{C}}^{M} \rightarrow \cpx{\tilde{C}}$ and
$\tilde{h}: \cpx{\tilde{C}} \rightarrow \cpx{\tilde{C}}[1]$ given by
\begin{align*}
\tilde{f}(x) &=
\begin{cases}
  x, & \quad \text{if } \alpha \in M^{0}, \\
  0, & \quad otherwise,
\end{cases} & \\
\tilde{g}(x) &= x, & \qquad x \in C_{\alpha}.\\
\tilde{h}(x) &=
\begin{cases}
  -d^{-1}_{\alpha,\beta}(x), & \quad \text{if } \beta \rightarrow \alpha \in M, \\
  0, & \quad \text{otherwise};
\end{cases} &
\end{align*}

With this notation we can now formulate the following lemma.
\begin{Lem}\label{lem:morse_contraction}
  The diagram
  \begin{center}
    \begin{tikzpicture}
      \matrix (n) [matrix of math nodes, column sep=2em,
      nodes={asymmetrical rectangle}]
      { \cpx{\tilde{C}}^{M} & \cpx{\tilde{C}} \\ };

      \path[-stealth] ($(n-1-1.east)+(0,0.1)$) edge node [auto] {$\tilde{g}$}
        ($(n-1-2.west)+(0,0.1)$);

      \path[-stealth] ($(n-1-2.west)-(0,0.1)$) edge node [auto] {$\tilde{f}$}
        ($(n-1-1.east)-(0,0.1)$);

      \draw[-stealth] (n-1-2) to [out=30, in=330, looseness=5]
        node [auto] {$\tilde{h}$} (n-1-2);
  \end{tikzpicture}

  \end{center}
  is a contraction.
\end{Lem}
\begin{proof}
  We first need to verify that $\tilde{f}$ and $\tilde{g}$ are chain
  maps, which is readily seen. Next we check the identities
  \[
  \tilde{f} \tilde{g} = 1, \quad
  \tilde{g} \tilde{f} = 1 + \tilde{d}\tilde{h} +\tilde{h}\tilde{d}.
  \]
  The first one is obvious, and the second follows from the fact that
  for a basis element $x \in C_{\alpha}$, $\tilde{d}\tilde{h}(x) = -x$
  if there is an edge $\beta \rightarrow \alpha$ in $M$, and 0
  otherwise; and similarly $\tilde{h}\tilde{d}(x) = -x$ if there is an
  edge $\alpha \rightarrow \beta$ in $M$, and
  0 otherwise. The identities
  \[
  \tilde{h}\tilde{g} = 0, \quad \tilde{f}\tilde{h} = 0, \quad \tilde{h}^2 = 0
  \]
  follow from that vertices in $M^{0}$ are not incident
  to any edge in $M$ (the first two) and that no vertex
  is incident to more than one edge in $M$ (the third).
\end{proof}

Let us now define the perturbation $t$ on $\cpx{\tilde{C}}$ as
$t = d- \tilde{d}$, where $d$ is the differential on $\cpx{C}$, so
\[
t(x) = \sum_{\alpha\rightarrow\beta \not\in M}d_{\beta,\alpha}(x)
\] for $x \in C_{\alpha}$. This makes
$\cpx{\tilde{C}}^{t}$ and $\cpx{C}$ isomorphic as based complexes.

\begin{Lem}
  The diagram
  \begin{center}
    \begin{tikzpicture}
      \matrix (n) [matrix of math nodes, column sep=2em,
      nodes={asymmetrical rectangle}]
      { \cpx{C}^{M} & \cpx{C} \\ };

      \path[-stealth] ($(n-1-1.east)+(0,0.1)$) edge node [auto] {$g$}
        ($(n-1-2.west)+(0,0.1)$);

      \path[-stealth] ($(n-1-2.west)-(0,0.1)$) edge node [auto] {$f$}
        ($(n-1-1.east)-(0,0.1)$);

      \draw[-stealth] (n-1-2) to [out=30, in=330, looseness=5]
        node [auto] {$h$} (n-1-2);
    \end{tikzpicture}
  \end{center}
  where, for $x \in C_{\alpha}$ with $\alpha \in I_n$,
  \begin{align*}
  d_{\cpx{C}^{M}}(x) &= \sum_{\beta \in M^{0} \cap I_{n-1}} \Gamma_{\beta,\alpha}(x)
  & \qquad
  f(x) &= \sum_{\beta \in M^{0} \cap I_{n}} \Gamma_{\beta,\alpha} (x) \\
  g(x) &=  \sum_{\beta \in I_{n}} \Gamma_{\beta,\alpha} (x)
  & \qquad
  h(x) &= \sum_{\beta \in I_{n+1}} \Gamma_{\beta,\alpha} (x)
  \end{align*}
  is a filtered contraction.
\end{Lem}
\begin{proof}
  From Lemma~\ref{lem:morse_contraction} together with the fact that there
  are no infinite paths in $\morsegraph{\cpx{C}}{M}$, the Morse graph of
  $\cpx{C}$, we can deduce that
  $ht$ is locally nilpotent, and we can thus invoke the perturbation lemma.
  It is not so hard to see that the perturbed differential on
  $\cpx{\tilde{C}}^{M}$ is given by
  \[
  d(x) = \sum_{i = 0}^{\infty} t(ht)^i(x)
  = \sum_{\beta \in M^{0}\cap I_{n-1}} \Gamma_{\beta,\alpha} (x)
  \]
  and the maps $f$, $g$ and $h$ by
  \begin{align*}
    f(x) &= \sum_{i = 0}^{\infty} f(ht)^i (x)
    = \sum_{\beta \in M^{0}\cap I_{n}} \Gamma_{\beta,\alpha} (x) \\
    g(x) &= \sum_{i = 0}^{\infty} g(ht)^i (x)
    = \sum_{\beta \in I_{n}} \Gamma_{\beta,\alpha} (x) \\
    h(x) &= \sum_{i = 0}^{\infty} (ht)^ih(x)
    = \sum_{\beta \in I_{n+1}} \Gamma_{\beta,\alpha} (x) \\
  \end{align*}
  where $x \in C_{\alpha}$.
\end{proof}

The above result is also shown (without the use of the perturbation
lemma) in \cite{berglund:hpt_morse} using a result from
\cite{joellenbeck_welker:morse}.

From the preceding lemma, the main result of algebraic Morse theory
now follows.
\begin{Thm}
  Let $\cpx{C}$ be a based complex with a Morse matching $M$, then
  there is a differential on the graded module
  $\bigoplus_{\alpha \in M^0} C_{\alpha}$ such that the resulting
  complex is homotopy equivalent to $\cpx{C}$.
\end{Thm}

\bibliographystyle{amsalpha}
\bibliography{bibfil}

\end{document}